\documentclass{amsart}
\usepackage{hyperref}

\hypersetup{
	colorlinks=true,
    linkcolor=black,
    citecolor=black,
    filecolor=black,
    urlcolor=black,
}
\usepackage{xy}
\input xy
\xyoption{all}
\newcommand{\scr}[1]{\mathscr{#1}}
\usepackage{Definitions}
\usepackage{amssymb}
\usepackage{Environments}
\usepackage{PageSetup}
\usepackage{tikz-cd}

\newcommand{\EH}{{\check{E}}}
\theoremstyle{definition}
\newtheorem{thmA}{Theorem}

\renewcommand{\theenumi}{\arabic{enumi}}

\makeatletter
\renewcommand{\p@enumii}{\theenumi.}
\makeatother
\begin{document}
	\title{On a nilpotence conjecture of J.P. May}
	\author{Akhil Mathew}
	\address{University of California, Berkeley, CA 94720-3840}
	\email{amathew@math.berkeley.edu}
	\urladdr{http://math.berkeley.edu/~amathew/}

	\author{Niko Naumann }
	\address{University of Regensburg\\
	NWF I - Mathematik; Regensburg, Germany}
	\email{Niko.Naumann@mathematik.uni-regensburg.de}
	\urladdr{http://homepages.uni-regensburg.de/~nan25776/}

	\author{Justin Noel}
	\address{University of Regensburg\\
	NWF I - Mathematik; Regensburg, Germany}
	\email{justin.noel@mathematik.uni-regensburg.de}
	\urladdr{http://nullplug.org}
	\thanks{The first author was partially supported by the NSF Graduate Research Fellowship under grant DGE-110640. The latter two authors were supported by the SFB 1085 - Higher Invariants, Regensburg.}

\date{\today}

\begin{abstract}
	We prove a conjecture of J.P.~May concerning the nilpotence of elements in
	ring spectra with power operations, i.e., $H_\infty$-ring spectra. Using an
	explicit nilpotence bound on the torsion elements in $K(n)$-local
	$H_\infty$-algebras over $E_n$,  we reduce the conjecture to the
	nilpotence theorem of Devinatz, Hopkins, and Smith. As
	corollaries we obtain nilpotence results in various bordism
	rings including $M\mathit{Spin}_*$ and $M\mathit{String}_*$, results about the behavior of the Adams spectral sequence for $E_\infty$-ring spectra, and the non-existence of $E_\infty$-ring structures on certain complex oriented ring spectra.
\end{abstract}

\maketitle

\section{Introduction} 
\label{sec:introduction}
	Understanding the stable homotopy of the sphere has been a driving motivation of algebraic topology from its very beginning. Early landmark results include Serre's theorem that every element in the positive stems is of finite order and Nishida's theorem that every element in the positive stems is (smash-)nilpotent. This was vastly generalized by the nilpotence theorem of Devinatz, Hopkins and Smith, which states that complex bordism is sufficiently fine a homology theory to detect nilpotence in general ring spectra. On the other hand, already Nishida's proof used basic geometric constructions, namely extended powers, to transform the additive information of Serre's theorem into the multiplicative statement of nilpotence. This was made much more systematic and general in the work of May and his collaborators on $H_\infty$-ring spectra, which in particular led to a specific nilpotence conjecture for this restricted class of ring spectra. In this note, we will establish his conjecture, as follows:

\begin{thmA}\label{thm:may-original}
	Suppose that $R$ is an $H_\infty$-ring spectrum and $x\in \pi_* R$ is in the kernel of the Hurewicz homomorphism $\pi_*R \rightarrow H_*(R;\bZ)$. Then $x$ is nilpotent. 
\end{thmA}
This result was conjectured by May  and verified under the additional hypothesis that $px=0$ for some prime $p$, in \cite[Ch.~II Conj.~2.7 \& Thm.~6.2]{BMMS86}. In the case $R$ is the sphere spectrum, \Cref{thm:may-original} is equivalent to Nishida's nilpotence theorem. In contrast to the nilpotence theorem of \cite{DHS88}, \Cref{thm:may-original} does not require any knowledge of the complex cobordism of $R$, but we must add the hypothesis that $R$ is $H_\infty$.

We will prove the following result which implies \Cref{thm:may-original}\footnote{An elementary argument shows that \Cref{thm:may-original} implies \Cref{thm:may} as well.}:
\begin{thmA}\label{thm:may}
	Suppose that $R$ is an $H_\infty$-ring spectrum and $x\in \pi_* R$ has nilpotent image, via the Hurewicz homomorphism, in $H_*(R;k)$ for $k=\mathbb{Q}$ and $k=\mathbb{Z}/p$ for each prime $p$. Then $x$ is nilpotent. 
\end{thmA}
Indeed, since each of the above Hurewicz homomorphisms factors through the integral Hurewicz homomorphism, we see that May's conjecture follows from \Cref{thm:may}.

The outline of this note is as follows: In \Cref{sec:pfofthm:may} we reduce \Cref{thm:may} to the nilpotence theorem using designer-made power operations due to Rezk (\Cref{lem:power-ops}). These operations give explicit nilpotence bounds in a $K(n)$-local context (\Cref{thm:nilpotence}) which complete the proof of the main theorem. In \Cref{sec:power-ops} we build on Strickland's foundational work on operations in Lubin--Tate theory to provide a proof of \Cref{lem:power-ops}. We conclude with several applications which should be of independent interest as well as some speculative refinements of \Cref{thm:may} in \Cref{sec:apps}. The applications include: 
\begin{enumerate}
	\item The nilpotence of elements in bordism rings (\Cref{prop:bordism,thm:vanishing,thm:converse}).
	\item Results about the behavior of the Adams spectral sequence for $E_\infty$-ring spectra (\Cref{prop:ass}).
	\item The non-existence of $E_\infty$-ring structures on certain complex oriented ring spectra (\Cref{prop:non-realizability}).
\end{enumerate}

For one of these applications we will need the following fact: If $R$ is an $E_2$-ring and $x\in \pi_k R$ then $R[x^{-1}]$ is canonically an $E_2$-ring. In \Cref{sec:localization} we provide a quick proof of the more general statement for $E_n$-rings, provided $n$ is at least 2. 

In a preliminary version of this paper \cite[Prop.~4.3]{MNN} we observed that \Cref{thm:may} implied half of an analogue in Morava $E$-theory of Quillen's $\mathcal{F}$-isomorphism \cite[Thm.~7.1]{Qui71b}. We meanwhile found an independent proof of the full result, and this will be documented elsewhere.


{\bf Acknowledgements.}	 
\Cref{thm:may-original} is originally due to Mike Hopkins, who has known this result for some time. We would like to thank him for his blessing in publishing our own arguments below. We would also like to thank Charles Rezk, Tyler Lawson, and Jacob Lurie for several enlightening discussions. Finally, we thank the anonymous referee for carefully reading this paper.

\section{The proof of \Cref{thm:may}}\label{sec:pfofthm:may}

Throughout this section, the notations and assumptions of \Cref{thm:may} are in force. 

Recall that for each prime $p$ and positive integer $n$, there are 2-periodic ring spectra\footnote{Usually $K(n)$ denotes a $2(p^n-1)$-periodic theory. The $2$-periodic version can be obtained by a faithfully flat extension of the standard $2(p^n-1)$-periodic theory. Both variants have identical Bousfield classes and either variant can be used to detect nilpotence.} $K(n)$ and $E_n$ which are related by a map $E_n\rightarrow K(n)$ of ring spectra inducing the quotient map of the 
local ring $\pi_0 E_n$ to its residue field $\pi_0K(n)$. The first family
consists of the  Morava $K$-theories, which play an important role in the
Ravenel conjectures \cite{Rav84} and are especially amenable to computation.
The second family consists of Lubin--Tate theories which satisfy certain
universal properties that make them extremely rigid; in particular, each
of them admits an essentially unique $E_\infty$-algebra structure and a corresponding theory of power operations; see \Cref{sec:power-ops} for more details.

By the nilpotence theorem \cite[Thm.~3.i]{HoS98} if we can show that $x$ is nilpotent in $H_*(R;\bQ)$, $H_*(R;\bZ/p)$, and $K(n)_*R$ for each prime $p$ and positive integer $n$, then $x$ is nilpotent. Since $x$ has nilpotent image in $H_*(R;\bQ)\cong \pi_*R\otimes \bQ$, by replacing $x$ with a suitable power we may assume it is torsion. Now since $x$ is torsion it is zero in $H_*(R;\bQ)$ and by assumption, $x$ is nilpotent in $H_*(R;\bZ/p)$ for each prime $p$. To show $x$ is nilpotent in $K(n)_*R$, we will show it is nilpotent in the ring $\pi_*L_{K(n)}(E_n\wedge R)$ and then map to $K(n)_*R$. So \Cref{thm:may} will follow from the following theorem, applied to $T=L_{K(n)}(E_n\wedge R)$ and the image of $x$ in $T$ under the $E_n$-Hurewicz map.


To simplify notation in what follows, we have put $E=E_n$, $\EH(X)=L_{K(n)}(E\wedge X)$, and $\EH_*(X)=\pi_*\EH(X)$. The `check' notation, e.g., $\EH$, is meant to remind the reader that we are working in a $K(n)$-local category.
In particular, the $T$ that appears in \Cref{thm:nilpotence} is $K(n)$-local;
see the next section for our conventions.

	\begin{thm}\label{thm:nilpotence}
		Suppose $T$ is a  $H_\infty$-$\EH$-algebra and $x\in \pi_j T$.	
		\begin{enumerate}
			\item \label{it:even-degrees} If $j$ is even and $p^m x=0$ then \[x^{(p+1)^m}=0.\]
			\item  \label{it:odd-degrees} If $j$ is odd then $x^2=0$.
		\end{enumerate}
	\end{thm}

	Our proof of this will depend on the following unpublished result of Rezk \cite[p.~12]{Rez10b} which we will prove in \Cref{sec:power-ops}.

	\begin{lemma}\label{lem:power-ops} 
		 Suppose $T$ is a  $H_\infty$-$\EH$-algebra. Then there are operations $Q$ and $\theta$ acting on $\pi_0 T$ and natural with respect to maps of $H_\infty$-$\EH$-algebras satisfying:
		\begin{enumerate}
			\item\label{enum:formula} $(-)^p=Q(-)+p\theta(-)$.
			\item\label{enum:additivity} $Q$ is additive.
			\item\label{enum:zero} $\theta(0)=0$. 
		\end{enumerate}
	\end{lemma}

	\begin{proof}[Proof of \Cref{thm:nilpotence}]
		The claim about odd degree elements is precisely
		\cite[Prop.~3.14]{Rez09}, so we may assume $j$ is even\footnote{Since
		$2x^2 = 0$ for $x$ in odd degrees, we could alternatively appeal to 
		\Cref{thm:nilpotence}.\ref{it:even-degrees} to conclude $x^6 = 0$. This weakening of \Cref{thm:nilpotence}.\ref{it:odd-degrees} suffices for proving \Cref{thm:may}, but misses Rezk's sharp bound.}. Since $\pi_* T$ is a $\pi_* E$-algebra, either the periodicity generator in $\pi_2 E$ has non-trivial image in $\pi_2 T$ or $\pi_*T=0$. In the latter case, the theorem holds vacuously, so by dividing by a suitable power of the periodicity generator, we can furthermore assume that $j$ is zero.

		It follows from the first two items in \Cref{lem:power-ops} that if $p$ were not a zero-divisor in $\pi_0 T$ then 
		\begin{align*}
		\theta(p^m x) &=p^{pm-1}x^p-p^{m-1}Q(x)\\
		&=p^{m-1}\left((p^{(p-1)m}-1)x^p + x^p - Q(x)\right )
		\end{align*}
		which when combined with
		\[
		p^m\theta(x) =p^{m-1}\left( x^p - Q(x) \right)
		\]
		yields
		\begin{equation}\label{eqn:power-op}
			\theta(p^m x)=x^p(p^{pm-1}-p^{m-1})+p^m\theta(x).
		\end{equation}
		
		To see that \eqref{eqn:power-op} also holds in rings with $p$-torsion,
		consider $x\in\pi_0T$ as a map $x\colon S^0\to T$. Since the target is an $H_\infty$-$\EH$-algebra, this map canonically extends, up to homotopy, through the free $H_\infty$-$\EH$-algebra on $S^0$:
		\begin{equation}\label{dia:universal}
		\xymatrix{
S^0 \ar[d]^\iota \ar[r]^{x} & T \\
\EH ( \bP S^0) \ar@{-->}_{P(x)}[ru]
		}
	\end{equation}
		Since $\EH_0(\bP S^0)$ is torsion free \cite[Thm.~1.1]{Str98}, \eqref{eqn:power-op} holds in $\EH_0 (\bP S^0)$ with $\iota$ in place of $x$. After applying $\pi_0$ to Figure \eqref{dia:universal}, $P(x)$ induces a ring map sending $Q(\iota)$ and $\theta(\iota)$ to $Q(x)$ and $\theta(x)$ respectively, so \eqref{eqn:power-op} holds in $\pi_0 T$.

		Now, since $p^m x=0$, by multiplying \eqref{eqn:power-op} by $x$ and using \Cref{lem:power-ops}.\ref{enum:zero} we see that $p^{m-1}x^{p+1}=0$. The theorem now follows by induction on $m$.
	\end{proof}

	\section{Power operations in Morava $E$-theory}\label{sec:power-ops}
	Before  proving \Cref{lem:power-ops}, we first recollect enough results about the theory of $E_\infty$- and $H_\infty$-algebras from \cite{BMMS86,EKMM97} to define their variants in $E$ and $\EH$-modules.

	Recall that the category of $E_\infty$-ring spectra is equivalent to the
	category of algebras over the monad \[\bP(-)=\bigvee_{n\geq 0}
	{\sO(n)}_+\wedge_{\Sigma_n} (-)^{\wedge n},\] acting on $S$-modules. Here
	$\sO(n)$ is the $n$th space of an $E_\infty$-operad, i.e., any operad weakly
	equivalent to the commutative operad such that $\sO(n)$ is a free
	$\Sigma_n$-space. The structure maps for the monad are derived from the
	structure maps for the operad in a straightforward way \cite[\S 11]{Rez97}.
	The category of such algebras forms a model category and any two choices of
	$E_\infty$-operad yield Quillen equivalent models \cite[Thm.~1.6]{GoH04}. In
	fact, any such category is Quillen equivalent to a category of strictly
	commutative $S$-algebras. 

	The monad $\bP$ descends to a monad on the homotopy category of $S$-modules and the category of $H_\infty$-ring spectra is the category of algebras for this monad. Such spectra admit all of the structure maps of $E_\infty$-ring spectra, but these maps only satisfy the required coherence conditions up to homotopy. There is a  forgetful functor from the homotopy category of $E_\infty$-ring spectra to $H_\infty$-ring spectra which endows each $E_\infty$-ring spectrum with power operations as defined below, see  \cite{JoN14} for more details.

	 As shown in \cite{GoH04,GoH05}, each Lubin--Tate theory $E$ admits an essentially unique $E_\infty$-structure realizing the $\pi_* E$-algebra $E_*(E)$. Applying standard results from \cite{EKMM97}, we see that after taking a commutative model for $E$, the category of $E$-modules is a topological symmetric monoidal model category with unit $E$ and smash product $\wedge_E$. The category of commutative $E$-algebras is Quillen equivalent to the category of $E_\infty$-algebras in this category, which are in turn equivalent to the category of algebras over the following monad:
	\[
	 \bP_E(-)=\bigvee_{n\geq 0} {\sO(n)}_+\wedge_{\Sigma_n} (-)^{\wedge_E n}.
	\]

	By \cite[Ch.~VIII Lem.~2.7]{EKMM97}, $\bP_E$ respects $K(n)$-equivalences and descends to a monad $\bP_{\EH}$ on the homotopy category of $K(n)$-local $E$-modules. The smash product in this category is the $K(n)$-localization of the smash product of $E$-modules, i.e., $X\wedge_{\EH} Y:= L_{K(n)}(X\wedge_E Y)$. We will call the category of algebras over $\bP_{\EH}$ the category of $H_\infty$-$\EH$-algebras. Since the equivariant natural equivalences \[(\EH(-))^{\wedge_{\EH} n}\cong \EH((-)^{\wedge n})\] induce a natural equivalence
	\begin{equation}\label{eqn:base-change}
		\bP_{\EH}(\EH(-))\cong \EH(\bP(-)),
	\end{equation}
	we see that if $R$ is an $H_\infty$-ring spectrum then $\EH(R)$ is an $H_\infty$-$\EH$-algebra and acceptable input for \Cref{thm:nilpotence}. 

	Given an $H_\infty$-$\EH$-algebra $T$, a map $x\colon S^0\rightarrow  T$
	of spectra (i.e., a map of $E$-modules $E \to T$), and an $\alpha \in \EH_0({B\Sigma_p}_+)$ we obtain an operation 
	\[
		Q_\alpha \colon \pi_0 T \rightarrow \pi_0 T
	\]
	defined by the following composite:
	\[
	Q_\alpha(x)\, :\, S^0\xrightarrow{\alpha} \EH({B\Sigma_p}_+)\cong \sO(p)_+ \wedge_{\Sigma_p} E^{\wedge_\EH p}\xrightarrow{D_p(x)} \sO(p)_+ \wedge_{\Sigma_p} T^{\wedge_\EH p} \xrightarrow{\mu_p}T.\] 
	Here $D_p$ is the functor associated to the $p$th extended power construction in $E$-modules and $\mu_p$ is the $H_\infty$-$\EH$ structure map on $T$. It is clear that, by construction, $Q_\alpha$ is natural in maps of $H_\infty$-$\EH$-algebras.

	\begin{example}\label{ex:pth-power}
		The inclusion of the base point into $B\Sigma_p$ and the
		$E$-Hurewicz homomorphism induce a map $i\colon S^0\rightarrow \EH({B\Sigma_p}_+)$. The associated operation is the $p$th power map. 
	\end{example}

	Since $\EH_*({B\Sigma_p}_+)$ is a finitely generated free $E_*$-module and concentrated in even degrees, we have a duality isomorphism \cite[Thm.~3.2]{Str98}: \[
	\EH_0({B\Sigma_p}_+)\cong \mathrm{Mod}_{\pi_0 E}(E^0({B\Sigma_p}_+), \pi_0 E).
	\]
	Therefore we can construct operations by defining the corresponding linear maps on $E^0({B\Sigma_p}_+)$.

	By an elementary diagram chase, the additive operations correspond to the subgroup $\Gamma$ of $\EH_0({B\Sigma_p}_+)$ defined by the following exact sequence:
	\[
	 0\rightarrow \Gamma \rightarrow \EH_0({B\Sigma_p}_+)\rightarrow \prod_{0<i<p}\EH_0((B\Sigma_i\times B\Sigma_{p-i})_+),
	 \] where the right hand map is the product of the transfer homomorphisms, compare \cite[\S 6]{Rez09}. To rephrase this in terms of cohomology, let $J$ be the ideal of $E^0({B\Sigma_p}_+)$ generated by the cohomological transfer maps. Then the additive operations correspond to those $\pi_0E$-module maps  $E^0({B\Sigma_p}_+)\rightarrow \pi_0 E$ which factor through the quotient $E^0({B\Sigma_p}_+)/J$.

		\begin{proof}[Proof of \Cref{lem:power-ops}]
		By \cite[Prop.~10.3]{Rez09} we have a commutative solid arrow diagram of $\pi_0 E$-algebras:
		\[
		\xymatrix{
E^0( {B \Sigma_p}_+) \ar[d]^{\varepsilon} \ar[r]^r & E^0({B\Sigma_p}_+)/J
 \ar@{.>}@/^1pc/[l]^s\ar[d]^{\phi_2} \\  
 \pi_0 E \ar[r]^{\phi_1} &  \pi_0 E/p
		}
		\]
		Here $\varepsilon$ is the map induced by the inclusion of a base point into $B\Sigma_p$. It is dual to the map $i$ from \Cref{ex:pth-power} and corresponds to the $p$th power operation $(-)^p$. The maps $r$ and $\phi_1$ are the obvious quotient maps, while $\phi_2$ is the unique map making the diagram commute.

		By applying \cite[Thm.~1.1]{Str98} once again, we know that $E^0({B\Sigma_p}_+)/J$ is a finitely generated free $\pi_0 E$-module. So the map $r$ admits a section $s$ of $\pi_0 E$-modules. By the discussion above, the composite map $\varepsilon\circ s\circ r$ determines an additive operation $Q$. Moreover
		\begin{align*}
			\phi_1 \circ \varepsilon \circ (\mathrm{Id}- s\circ r) &= \phi_2\circ r\circ (\mathrm{Id}-s\circ r)\\
			 &=\phi_2 \circ (r-r)\\
			 &=0.
		\end{align*}
		It follows that 
		\begin{equation}\label{eq:dual-formula}
			\epsilon -\epsilon\circ s\circ r= p \cdot f
		\end{equation} 
		for some homomorphism $f\colon E^0(B\Sigma_p)\rightarrow E^0$. If we let $\theta$ be the operation corresponding to $f$, then Lemmas \ref{lem:power-ops}.\ref{enum:formula} and \ref{lem:power-ops}.\ref{enum:additivity} will follow from \eqref{eq:dual-formula} and our definitions.

		 To prove \Cref{lem:power-ops}.\ref{enum:zero}, suppose $x=0\in \EH_0 R$. Then the extended power $D_p(x)$ factors through \[{E\Sigma_p}_+\wedge_{\Sigma_p} *^{\wedge_E p}\simeq *\] and $Q_\alpha(0)=0$ for all $\alpha$. 
	\end{proof}

\section{Applications}\label{sec:apps}

We collect some applications of our main result in this
section. Although all results can be applied toward $H_\infty$-ring spectra, we state them in terms of $E_\infty$-ring spectra since that is the case of greatest interest.

\begin{prop}\label{prop:odd-degrees}
	Suppose $x\in \pi_* R$ is an \emph{odd degree} element in the homotopy
	groups of an $E_\infty$-ring spectrum $R$ which has nilpotent image in $H_*(R;\bZ/2)$. Then $x$ is nilpotent.
\end{prop}
\begin{proof}
	By \Cref{thm:may} it suffices to show that $x$ has square zero image in
	rational and mod-$p$ homology for $p$ odd. This is a consequence of the fact
	that the homotopy groups of a commutative ring spectrum are graded-commutative, so odd degree elements square to 2-torsion by the sign rule. 
\end{proof}

\begin{prop}\label{prop:torsion-rings-are-dissonant}
	Suppose $R$ is an $E_\infty$-ring spectrum such that $0=m\cdot 1\in \pi_0 R$
	for some $m\neq 0$. Then, for every prime $p$ and non-negative integer $n$, $K(n)_* R=0$.
\end{prop}
\begin{proof}
	By \Cref{thm:nilpotence} applied with $T=\EH_n(R)$ , the image of $1$ in $K(n)_* R$ is nilpotent. 
\end{proof}

\begin{remark}
	The previous proposition immediately implies the following observation of
	Lawson: Any finite $E_\infty$-ring spectrum $R$ either has type 0,
	i.e., $H_*(R;\bQ)\neq 0$, or is weakly contractible. Indeed, if $H_*(R;\bQ)\cong \pi_*R\otimes \bQ=0$ then \Cref{prop:torsion-rings-are-dissonant} implies $K(n)_*R=0$ for every prime $p$ and non-negative integer $n$. Since the two-periodic Morava $K$-theory spectrum splits as a wedge of suspensions of the $2p^n-2$-periodic Morava $K$-theory spectrum $\overline{K}(n)$, we see that $\overline{K}(n)_*(R)=0$. 

	Since $R$ is finite, the Atiyah-Hirzebruch spectral sequence for $\overline{K}(n)_*R$ collapses to $H_*(R;\bZ/p)[v_n^{\pm 1}]$ for $n\gg 0$. This forces $H_*(R;\bZ/p)=0$ and hence $R_p$, the $p$-completion of $R$, is weakly contractible for all primes $p$. This combined with the fact that $\pi_* R$ is torsion implies that $R$ is weakly contractible.
\end{remark} 

\subsection{Applications to bordism}
	Recall that $BGL_1S$ is the classifying space for stable spherical
	fibrations \cite{May77}. This is an infinite loop space and associated to
	any infinite loop map $G\rightarrow GL_1 S$ is an ${E}_{\infty}$ Thom
	spectrum $MG$ \cite[Ch.~IX]{LMS86}.  Standard examples include the
	$J$-homomorphisms from $\mathit{SO}$, $\mathit{Spin}$, $\mathit{String}$ and
	their complex analogues. The homotopy groups of these geometric Thom spectra
	correspond to the bordism rings of the corresponding categories of (compact) manifolds \cite[Ch.~IV]{Rud98}. In this language, if a $G$-manifold $M$ represents a bordism class $[M]\in \pi_* MG$, then the nilpotence of $[M]$ is equivalent to the statement that the cartesian powers $M^n$ bound $G$-manifolds for all sufficiently large $n$.

	Now all $G$-manifolds are oriented if and only if $MG$ is $\bZ$-orientable. A choice of orientation determines a Thom isomorphism $H_*(BG;\bZ)\cong H_*(MG;\bZ)$. Even if $G$ does not arise from one of our geometric examples, $MG$ is $\bZ$-orientable if and only if $\pi_0 MG\cong \bZ$, and otherwise one has $\pi_0 MG\cong \bZ/2$ \cite[Ch.~IX, Prop.~4.5]{May77}. The latter case occurs if and only if the map $f\colon G\rightarrow GL_1 S$ does not lift to  $SL_1 S$, the connected component of the identity in $GL_1S$. 

	In the oriented case, we have Thom isomorphisms $ H_*(BG;k)\cong H_*(MG;k)$ for any field $k$. In the geometric examples, this happens when the classifying map \[f\colon G\rightarrow O\] lifts to $\mathit{SO}$. In these cases, the Thom isomorphism can be used to characterize the Hurewicz image of a $G$-bordism class $[M]$ in $H_*(MG;k)$ in terms of the $k$-characteristic numbers of $M$ (at least if $G$ is of finite type). These characteristic numbers can be calculated by pairing the fundamental class of $M$ with the characteristic classes of the stable normal bundle\footnote{More precisely, we are thinking of the normal characteristic numbers. Since the normal characteristic classes can be written in terms of tangential characteristic classes and vice versa, one set of numbers determines the other. In particular, all of the normal characteristic numbers vanish if and only if all of the tangential characteristic numbers vanish.} of $M$ \cite[p.401--402]{Swi02}. In particular, $[M]$ has trivial image in $H_*(MG;k)$ if and only if the $k$-characteristic numbers of $M$ vanish.

	In many important cases, these characteristic numbers have names. For example, when $M$ is a manifold, not necessarily oriented, it still admits a fundamental class in $k=\bZ/2$-homology and the $\bZ/2$-characteristic numbers of $M$ are the \emph{Stiefel-Whitney numbers}. These numbers determine the image of $[M]$ in $H_*(MO;\bZ/2)$. Stably complex manifolds, i.e., $U$-manifolds, have fundamental classes in integral homology and their images in $H_*(MU;\bZ)$ are described in terms of their \emph{Chern numbers}. Oriented manifolds also have \emph{Pontryagin numbers} which determine their image in $H_*(MSO;\bZ[1/2])\hookrightarrow H_*(MSO;\bQ)$. 

	The previous discussion, \Cref{thm:may} and \Cref{prop:odd-degrees} now imply the following:
	\begin{prop}\label{prop:bordism}
		Suppose $G\rightarrow \mathit{SO}$ is a map of infinite
		loop spaces and $G$ is of finite type (e.g., the 2-connected cover $G=\mathit{Spin}$ or the 6-connected cover $\mathit{String}$). If $M$ is a $G$-manifold whose rational and $\bZ/p$-characteristic numbers vanish for all primes $p$, then $M^n$ bounds a $G$-manifold for all sufficiently large $n$. 

		Moreover, if $M$ is an odd-dimensional manifold, then it suffices that the $\bZ/2$-characteristic numbers vanish.
	\end{prop}

	For the sake of completeness, we note that the unoriented case is simpler. 
	\begin{prop}\label{prop:unoriented}
		Suppose that $f\colon G\rightarrow GL_1 S$ is a map of based homotopy commutative $H$-spaces and non-trivial on $\pi_0$. Then $MG$ is a homotopy commutative ring spectrum such that the mod-2 Hurewicz homomorphism \[\pi_* MG\rightarrow H_*(MG;\bZ/2)\] is a split injection. In particular an element in $\pi_*MG$ is nilpotent if and only if its mod-2 Hurewicz image is nilpotent.
	\end{prop}
	\begin{proof}
		It follows from the results of \cite[Ch.~IX]{LMS86} that $MG$ is a homotopy commutative ring spectrum and by the discussion above $\pi_0 MG=\bZ/2$. So by \cite[Thm.~1.1]{W86}, $MG$ is an $H\bZ/2$-module and the Hurewicz homomorphism splits.
	\end{proof}

	To understand how these results fit in with classical bordism ring calculations, we offer the following two results:
	\begin{thm}\label{thm:vanishing}\ 
	\begin{enumerate}
	\item\label{it:vanishing-real}
	For a {\it String}-- (resp.~{\it Spin}--)manifold $M$, the following are equivalent:
	\begin{enumerate}
	\item  For all sufficiently large $n$, $M^n$ bounds a {\it String}-- (resp.~\mbox{{\it Spin}--)}manifold.
	\item\label{it:oriented} $M$ bounds an oriented manifold.
	\item\label{it:sw-pontryagin} All Stiefel--Whitney and Pontryagin numbers of $M$ vanish.
	\end{enumerate}
	\item\label{it:vanishing-complex}
	For a $U\langle 6 \rangle$-- (resp.~$\mathit{SU}$--)manifold $M$, the following are equivalent:
	\begin{enumerate}
	\item For all sufficiently large $n$, $M^n$ bounds a $U\langle 6\rangle$-- (resp.~\mbox{$\mathit{SU}$--)}manifold.
	\item \label{it:complex} $M$ bounds a stably complex manifold.
	\item \label{it:chern} All Chern numbers of $M$ vanish.
	\end{enumerate}
	\end{enumerate}
	

	\end{thm}

	\begin{thm}\label{thm:converse}
		Suppose that $R$ is one of the following Thom spectra: $MO$, $M\mathit{SO}$, $M\mathit{Spin}$, $M\mathit{String}$, $MU$, $M\mathit{SU}$, or $MU\langle 6\rangle$. Then the kernel of the integral Hurewicz homomorphism $\pi_* R\rightarrow H_*(R;\bZ)$ is precisely the ideal of nilpotent elements, i.e., the converse of \Cref{thm:may-original} also holds
		for these $R$. 
	\end{thm}

	In general the converse of \Cref{thm:may-original} holds when the integral homology ring is reduced, so every nilpotent element in the homotopy ring is in the kernel of the Hurewicz map. For example, the case of $MO$ in \cref{thm:converse} follows from \Cref{prop:unoriented} and the identification of $H_*(MO;\bZ/2)\cong H_*(BO;\bZ/2)$ as a polynomial algebra.  For the remainder of the claims in \Cref{thm:converse} we will show (see \Cref{fig:real,fig:complex}) that the relevant integral homology ring is a subring of a reduced ring. 

	In the remainder of this subsection, we will simultaneously prove \cref{thm:vanishing} and \cref{thm:converse} by analyzing the Hurewicz homomorphisms. First we note that the equivalences \eqref{it:oriented} $\iff$ \eqref{it:sw-pontryagin} and \eqref{it:complex} $\iff$ \eqref{it:chern} are classical:
	\begin{prop}\label{prop:classical}
		Suppose that $M$ is an oriented manifold and $N$ is a stably complex manifold. Then
		\begin{enumerate}
			\item $M$ bounds an oriented manifold if and only if the Stiefel-Whitney and Pontryagin numbers of $M$ vanish \cite[Cor.~1]{Mil60}. 
			\item $N$ bounds a stably complex manifold if and only if the Chern numbers of $N$ vanish \cite[Cor.~20.26]{Swi02}. 
		\end{enumerate}
	\end{prop}

	Now we consider the assertion about nilpotence in part \eqref{it:vanishing-real} of \Cref{thm:vanishing}. For this, we consider the following commutative diagram of graded commutative rings:
	\begin{figure}[H]
	\begin{center}
	\[
		\begin{tikzcd}
			\pi_*M\mathit{String} \rar \dar &  H_*(M\mathit{String};\bZ) \rar[hook] \dar[hook]
			 & H_*(M\mathit{String};\bZ/2)\times H_*(M\mathit{String};\bQ) \dar[hook]\\
			 \pi_*M\mathit{Spin} \rar \dar & H_*(M\mathit{Spin};\bZ) \rar[hook] \dar[hook]& H_*(M\mathit{Spin};\bZ/2)\times H_*(M\mathit{Spin};\bQ) \dar[hook] \\
			 \pi_*M\mathit{SO} \rar[hook] & H_*(M\mathit{SO};\bZ) \rar[hook] & H_*(M\mathit{SO};\bZ/2)\times H_*(M\mathit{SO};\bQ).
		\end{tikzcd}
	\]
	\end{center}
	\caption{\label{fig:real}The Hurewicz homomorphisms for $\pi_*M\mathit{String}$, $\pi_*M\mathit{Spin}$, and $\pi_*M\mathit{SO}$}
	\end{figure}
	In \Cref{fig:real} the vertical maps are induced by forgetting structure, the horizontal maps on the left are the integral Hurewicz homomorphisms, and the horizontal maps on the right are the product of the mod-2 reduction maps and the rationalization maps.  The diagram commutes by the naturality of the Hurewicz homomorphisms. 
	
	\begin{lemma}\label{lem:injective}
		The maps in \Cref{fig:real} labeled with hooked arrows are injective. In particular, since the bottom right term is reduced, so are all of the displayed subrings. As a consequence, every nilpotent element in the displayed bordism groups maps to zero under the integral Hurewicz map.
	\end{lemma}
	
	This lemma and \Cref{thm:may-original} immediately imply that the kernels of the integral Hurewicz maps in \Cref{fig:real} are precisely the nilpotent elements. It also immediately follows that these classes are precisely the elements which map to zero in $\pi_*M\mathit{SO}$.
	
	\begin{proof} (of \cref{lem:injective})
		Since all of these Thom spectra are orientable, the homology of these Thom spectra are isomorphic as rings to the homology of their corresponding classifying spaces. All of these spaces are of finite type, so their homology is dual to their cohomology when working with field coefficients. We will use this fact repeatedly below.

		We begin with the column on the right. The induced maps in mod-2 homology are injections by Stong's analysis of the associated Serre spectral sequences \cite{Sto63}. An easy argument with the Serre spectral sequence shows that the maps in rational homology are injections as well. 

		Since all of our Thom spectra are of finite type, the injectivity of the horizontal homomorphisms on the right is equivalent to the claim that the only torsion in the integral homology groups is simple 2-torsion. We first show that all 2-torsion is simple which is equivalent to the claim that the mod-2 Bockstein spectral sequence collapses at $E_2$. In the case of $M\mathit{SO}$, the calculation of the Bockstein action in \cite[p.~513]{Swi02} shows that the $E_2$-page of the Bockstein spectral sequence is concentrated in even degrees and hence collapses. For $M\mathit{Spin}$ and $M\mathit{String}$, this follows from the given injections into $H_*(M\mathit{SO};\bZ/2)$ and the naturality and convergence of the Bockstein spectral sequences. Altogether this implies that the maps in homology with $\mathbb{Z}_{(2)}$-coefficients are injective.

		To see that the there is no odd primary torsion, we recall that the composite map of spectra \[ KO \rightarrow KU \rightarrow KO \] corresponding to complexifying real vector bundles and then forgetting the complex structure is an equivalence after inverting two. It follows that there is a similar retraction between the zeroth spaces of their higher connective covers. This implies that $H_*(BO\langle k\rangle;\bZ[1/2])$ is a retract of $H_*(BU\langle k \rangle;\bZ[1/2])$. It is well known that latter groups are torsion free if $k\leq 4$, in which case $BU\langle k\rangle$ is either $BU$ or $B\mathit{SU}$. By \cite[Middle of p.~3485]{HoR95}, $H^*(B\mathit{String};\bZ/p)$ is concentrated in even degrees for $p$ odd and $H^*(BU\langle 6\rangle;\bZ/p)$  is concentrated in even degrees for all primes $p$. By applying the Bockstein spectral sequence again we see that $H_*(B\mathit{String};\bZ[1/2])$ and $H_*(BU\langle 6\rangle; \bZ)$ are torsion free. 

		The vertical homomorphisms in the middle are now injective due to the commutativity of the diagram. The bottom composite horizontal homomorphism can be calculated by calculating the Stiefel--Whitney and Pontryagin numbers of the manifold. This map is an injection by the first half of \Cref{prop:classical}.
\end{proof}

This completes the proof of the first half of \cref{thm:vanishing} and \cref{thm:converse}. To conclude this subsection, we address the complex analogs. The relevant diagram of homology rings in this case is:
	\begin{figure}[H]
	\begin{center}
	\[
		\begin{tikzcd}
			\pi_*MU\langle 6\rangle  \rar \dar &  H_*(MU\langle 6\rangle;\bZ) \rar[hook] \dar[hook]
			 & H_*(MU\langle 6\rangle;\bQ) \dar[hook]\\
			 \pi_*M\mathit{SU} \rar \dar & H_*(M\mathit{SU};\bZ) \rar[hook] \dar[hook]& H_*(M\mathit{SU};\bQ) \dar[hook] \\
			 \pi_*MU \rar[hook] & H_*(MU;\bZ) \rar[hook] & H_*(MU;\bQ).
		\end{tikzcd}
		\]
	\end{center}
	\caption{\label{fig:complex}The Hurewicz homomorphisms for $\pi_*M\mathit{U}\langle 6\rangle$, $\pi_*M\mathit{SU}$, and $\pi_*M\mathit{U}$}
	\end{figure}
	
	The bottom left map is injective by the second half of \Cref{prop:classical}. The maps in rational homology are injections by an easy argument with the Serre spectral sequence.  The injectivity of the horizontal arrows on the right is addressed during the proof of \cref{lem:injective}. Now one completes the proof of \Cref{thm:vanishing} using \cref{thm:may-original} as before.
	
\subsection{Differentials in the Adams spectral sequence and non-existence of $E_\infty$-structures}

We can use our main result to establish differentials in the Adams spectral sequence, as follows: 

	\begin{prop}\label{prop:ass}
	      Suppose $R$ is a bounded below $E_\infty$-ring spectrum such that $H_*(R;\bZ/p)$ is of finite type. Let $x$ be an element in positive filtration in the $H\bZ/p$-based Adams spectral sequence converging to the homotopy of the $p$-completion $R_p$ of $R$. Then
	      either 
	      \begin{enumerate}
	     	 \item $x$ does not survive the spectral sequence,
		 \item $x$ detects a non-trivial element in $\pi_* R_p\otimes \bQ$, or 
		 \item $x$ detects a nilpotent element in $\pi_*R_p$ and
		 as a consequence all sufficiently large powers of $x$
		 do not survive the spectral sequence.
	      \end{enumerate}
	\end{prop}

\begin{proof} 
	First we note that the hypotheses on $R$ guarantee the convergence of the Adams spectral sequence and hence any element in $\pi_* R_p$ is detected somewhere in the spectral sequence. If our specified element $x$ fails the first two properties, then it is a permanent cycle and detects a torsion element $z\in\pi_*R_p$. Since $x$ is in positive filtration, $z$ has trivial mod $p$ Hurewicz image and is therefore nilpotent by \Cref{thm:may}. Thus, for every sufficiently large $n$, the element $x^n$ is a permanent cycle detecting $z^n=0$ in homotopy. Such an element can not survive the spectral sequence.
\end{proof}

	\Cref{thm:may} also implies the non-realizability of certain $E_\infty$-ring spectra. To precisely state this result, recall that there are non-nilpotent elements called $v_n\in \pi_{2(p^n-1)}MU$ for each positive $n$ and prime $p$. There are many choices for such elements, but for our purposes we can take any such element detected in positive Adams filtration. 

	\begin{prop}\label{prop:non-realizability}
	  Let $R$ be a bounded below ring spectrum (up to homotopy) under $MU$ such that for some prime $p$ the image of $v_n^k$ in $\pi_*R_p$ is non-nilpotent $p$-torsion for some positive integers $n$ and $k$. Moreover suppose that $H_*(R;\bZ/p)$ is of finite type. Then $R$ does not admit the structure of an $E_\infty$-ring spectrum.
	\end{prop}
	\begin{proof} 
		If $R$ admitted an $E_\infty$-ring structure then so would its $p$-completion $R_p$. So assume $R_p$ is $E_\infty$. Since maps of spectra never lower Adams filtration, the image of $v_n$ in $\pi_* R_p$ must be detected in positive Adams filtration. Since this element is torsion and non-nilpotent, its existence is a contradiction to \Cref{prop:ass}.
	\end{proof}

	\begin{example}
	  The previous result implies that many ring spectra such as $R=MU/(p^i)$, $BP/(p^i v_n^k)$, or $ku/(p^i \beta^k)$, where $\beta$ is the Bott element and $i$ and $k$ are positive integers, do not admit $E_\infty$-ring structures. 
	\end{example}

\subsection{Conceivable refinements of \Cref{thm:may}}

In deducing nilpotence in the homotopy of ring spectra from homological assumptions, there
is an obvious tension between the class of ring spectra to allow and the homology theories used
to test for nilpotence. On one extreme, the nilpotence theorem works for
general ring spectra but needs the more sophisticated homology theory
$MU$ to test against. Somewhat on the other extreme, \Cref{thm:may} applies
only to $H_\infty$-ring spectra, but only needs the most elementary homology theories to test against.

An approximately intermediate result will be derived from \Cref{thm:hopkinsmahowald}, which is an 
unpublished result of Hopkins and Mahowald. Before proving this we will now check that homotopy colimits of connective algebras are connective.

We will use the language of $\infty$-categories \cite{Lur09} and the formalism of
$\infty$-operads \cite[Ch.~2]{Lur12} in this subsection. We note that the homotopy theory of $\infty$-operads has been shown to be Quillen equivalent to that of dendroidal sets \cite{HHM} which is, in turn, Quillen equivalent to the homotopy theory of colored simplicial operads \cite{CiM13}. 

In this context, the symmetric monoidal smash product functor on spectra makes
the associated $\infty$-category $\mathrm{Sp}$ into a symmetric monoidal $\infty$-category such that the smash product commutes with (homotopy) colimits in each variable \cite[Prop.~4.1.3.10]{Lur12}. Using this symmetric monoidal structure we obtain categories $\mathrm{Alg}_{\mathcal{O}}(\mathrm{Sp})$ of $\mathcal{O}$-algebras in $\mathrm{Sp}$ for every $\infty$-operad $\mathcal{O}$ \cite[\S 2.1.3]{Lur12}. For example, if $0\leq k< \infty$, the little $k$-cubes operad defines a topological category of operators whose operadic nerve is the $\infty$-operad  $E_k$ \cite[Defn.~5.1.0.2]{Lur12}. This $\infty$-operad corresponds to a simplicial operad with one color such that the $\infty$-category of $E_k$-ring spectra is the $\infty$-category $\mathrm{Alg}_{E_k}(\mathrm{Sp})$ \cite[Defn.~8.1.0.1]{Lur12}.

Since the smash product of two connective spectra is again connective, the smash product functor restricts to a symmetric monoidal structure on the subcategory $\mathrm{Sp}_{\geq 0}$ of connective spectra \cite[Ex.~2.2.1.3]{Lur12}.  This makes $\mathrm{Alg}_{\mathcal{O}}(\mathrm{Sp}_{\geq 0})$ into a subcategory of $\mathrm{Alg}_{\mathcal{O}}(\mathrm{Sp})$. 

\begin{lemma} \label{lem:connective}
Let $\mathcal{O}$ be an $\infty$-operad with one color and 
let $\mathrm{Alg}_{\mathcal{O}}( \mathrm{Sp})$ be the $\infty$-category of
$\mathcal{O}$-algebras in spectra. Then the subcategory
$\mathrm{Alg}_{\mathcal{O}}(
\mathrm{Sp}_{\geq 0}) \subseteq \mathrm{Alg}_{\mathcal{O}}( \mathrm{Sp})$
spanned by $\mathcal{O}$-algebras in connective spectra
is closed under colimits.
\end{lemma}
\begin{proof} 
There is a 
free-forgetful adjunction \cite[Prop.~3.1.3.11]{Lur12}
\begin{equation} \label{adj} (\mathbb{P}_{\mathcal{O}},
\mathrm{U}_{\mathcal{O}}) \colon \mathrm{Sp}
\rightleftarrows \mathrm{Alg}_{\mathcal{O}}( \mathrm{Sp}). \end{equation}
Here $\mathrm{U}_{\mathcal{O}}$ is the functor that sends an $\mathcal{O}$-algebra to
its underlying spectrum. It is conservative \cite[Lem.~3.2.2.6]{Lur12} and commutes with sifted colimits \cite[Prop.~3.2.3.1]{Lur12}, so it is monadic \cite[Thm.~6.2.2.5]{Lur12}. The functor $\mathbb{P}_{\mathcal{O}}$ is described via
$\mathbb{P}_{\mathcal{O}}(X) \simeq \bigsqcup_{n \geq
0} \left( \mathcal{O}(n)_+ \wedge X^{\wedge n}\right)_{h \Sigma_n} $ by \cite[Prop.~3.1.3.11]{Lur12} since $\mathcal{O}$ has one color. Since $\mathrm{Sp}_{\geq 0}$ is closed under smash powers, smashing with the spaces $\mathcal{O}(n)_+$, and colimits, we see that $\mathbb{P}_{\mathcal{O}}$ takes $\mathrm{Sp}_{\geq 0}$ to $\mathrm{Alg}_{\mathcal{O}}(
\mathrm{Sp}_{\geq 0}) \subset \mathrm{Alg}_{\mathcal{O}}( \mathrm{Sp})$. 
Consider the composition \[T_{\mathcal{O}} =\mathbb{P}_{\mathcal{O}}\circ
\mathrm{U}_{\mathcal{O}}
\colon
\mathrm{Alg}_{\mathcal{O}}( \mathrm{Sp}) \to
\mathrm{Alg}_{\mathcal{O}}(\mathrm{Sp}).\]
Since the adjunction \eqref{adj} is monadic, 
any $X \in \mathrm{Alg}_{\mathcal{O}}(\mathrm{Sp})$ can be obtained
as the geometric realization of the simplicial bar construction
$B(T_{\mathcal{O}}, X)_\bullet$, where $B(T_{\mathcal{O}},
X)_\bullet$ is a simplicial object in 
$\mathrm{Alg}_{\mathcal{O}}(\mathrm{Sp})$ with 
$B(T_{\mathcal{O}},
X)_n = T_{\mathcal{O}}^{n+1} X$ \cite[Prop.~6.2.2.12]{Lur12}. 

Let $\mathcal{I}$ be an $\infty$-category and let $F \colon \mathcal{I} \to
\mathrm{Alg}_{\mathcal{O}}( \mathrm{Sp}_{\geq 0})$ be a functor. 
We can use the bar construction to compute $\varinjlim_{\mathcal{I}} F $. 
Namely, consider the functor $\widetilde{F}_\bullet\colon \mathcal{I} \to \mathrm{Fun}( \Delta^{op},
\mathrm{Alg}_{\mathcal{O}}( \mathrm{Sp}))$ given by
$B(T_{\mathcal{O}}, F)_\bullet$. 
Then $F \simeq |\widetilde{F}_\bullet|$ as functors $\mathcal{I} \to
\mathrm{Alg}_{\mathcal{O}}( \mathrm{Sp})$. Therefore, 
we have
\begin{equation} \label{colimiteqn}\varinjlim_{\mathcal{I}} F   =
\varinjlim_{\mathcal{I}} | \widetilde{F}_\bullet|
\simeq \left| \varinjlim_{\mathcal{I}}B( T_{\mathcal{O}}, F)_\bullet \right| 
\end{equation}
By definition, $\varinjlim_{\mathcal{I}}B(T_{\mathcal{O}}, F)_\bullet $
refers to the simplicial $\mathcal{O}$-algebra $A_\bullet$ with $A_n = 
\varinjlim_{\mathcal{I}}
T_{\mathcal{O}}^{n+1}F$, where the colimit
is computed in $\mathrm{Alg}_{\mathcal{O}}( \mathrm{Sp})$. 
Since $\mathbb{P}_{\mathcal{O}}$ is a left adjoint, it commutes with colimits and 
we can also write this as $A_n \simeq \mathbb{P}_{\mathcal{O}} \left(
\varinjlim_{\mathcal{I}} \mathrm{U}_{\mathcal{O}}  T_{\mathcal{O}}^n F\right)$. 
Since $\mathrm{Sp}_{\geq 0} \subset \mathrm{Sp}$ is closed under colimits, and
$\mathbb{P}_{\mathcal{O}}$ preserves connectivity, it
follows that 
each $A_n \in \mathrm{Alg}_{\mathcal{O}}(
\mathrm{Sp})$ is connective. 
Since the forgetful functor $\mathrm{U}_{\mathcal{O}}$ commutes with
geometric realizations, \eqref{colimiteqn} now shows that
$\varinjlim_{\mathcal{I}} F  $ is connective as desired. 
\end{proof} 

\begin{thm}[Hopkins-Mahowald]\label{thm:hopkinsmahowald} 
	For every prime $p$, the free $E_2$-ring $R$ with $p=0$ is the Eilenberg--MacLane spectrum $H\bZ/p$.
\end{thm}

The following argument has also appeared in the preprint of Antol{\`i}n-Camarena
and Barthel \cite{ACB14}. 
\begin{proof}
	First we recall how $R$ is constructed. Let $\bP_2(-)$ (resp.~$\bP_{2,p}(-)$) denote the free $E_2$-ring functor on spectra (resp.~$H\bZ/p$-modules). Given an $E_2$-ring spectrum $T$ and a map $f\colon X\rightarrow T$ of spectra, let $\tilde{f}$ denote the adjoint map $\bP_2 X\rightarrow T$ of $E_2$-rings. Now $R$ is given as the following homotopy pushout diagram of $E_2$-ring spectra:
	\[ \xymatrix{  \bP_2 (S^0)\ar[r]^{\tilde{p}}\ar[d]^{\tilde{0}} & S^0\ar[d]\\
	 S^0 \ar[r] & R.}\]
 
	 First we observe that $R$ is connective, by
	 Lemma~\ref{lem:connective}. 	 Next we claim that $R$ is $p$-complete. By
	 construction $\pi_*R$ is a graded $\bZ/p$-algebra and hence $\pi_n R$ is a
	 $\bZ/p$-module for each integer $n$. It follows that \[ \mathrm{Hom}(\bZ[1/p],\pi_n R)=\mathrm{Ext}(\bZ[1/p], \pi_n R)=0\] for each $n$, so $R$ is $p$-complete by \cite[Prop.~2.6]{Bou79}. 

	 By construction, $R$ admits a canonical $E_2$-ring map $f\colon R\rightarrow H\bZ/p$ extending the unit map $S\rightarrow H\bZ/p$. Now $f$ is a map between two $p$-complete spectra and hence an equivalence if and only if it induces an isomorphism in mod-$p$ homology \cite[Thm.~3.1]{Bou79}. The homology of $H\bZ/p$ is the dual Steenrod algebra $\mathcal{A}_*$. To compute the homology of the source we smash the defining homotopy pushout diagram for $R$ with $H\bZ/p$ and apply $\pi_*$. 

	 After applying the natural equivalence $H\bZ/p\wedge \bP_2(-)\cong \bP_{2,p}(H\bZ/p\wedge -)$, we obtain the following homotopy pushout diagram of $E_2$-rings in $H\bZ/p$-modules:
	\[ \xymatrix{  \bP_{2,p} (H\bZ/p)\ar[r]^{\tilde{p}}\ar[d]^{\tilde{0}} & H\bZ/p\ar[d]\\
	 H\bZ/p \ar[r] & H\bZ/p  \wedge R}\]
	 Since $p\simeq 0$ in $H\bZ/p$-modules, we see that \[H\bZ/p\wedge R\simeq \bP_{2,p}(H\bZ/p \wedge S^1)\cong H\bZ/p \wedge \bP_2(S^1).\]
	 Now the generalized Snaith splitting theorem shows that $\bP_2(S^1)\simeq \Sigma^\infty_+ \Omega^2 \Sigma^2 S^1$ \cite[\S 6]{May09}, \cite[VII\S 5]{LMS86}, \cite[Thm.~6.1]{May72}. Using this splitting and two applications of the Serre spectral sequence we see that $H_*(R;\bZ/p)$ has the same Poincare series as $\mathcal{A}_*$, so it suffices to show that $H_*f$ is surjective.  

	  The remainder of the argument is essentially that of \cite[Ch.~III Prop.~4.8]{BMMS86}. Since $H_*f$ is a map of graded commutative rings, it is surjective if it surjects onto the indecomposables of $\mathcal{A}_*$. By \cite[Ch.~III Thms.~2.2-2.3]{BMMS86}, the indecomposables of $\mathcal{A}_*$ are generated by the Bockstein class in degree one under the action of the Dyer-Lashof operations coming from the underlying $E_2$-ring structure. Since $H_* f$ commutes with these operations it suffices to show that it hits the Bockstein class in degree one. 

	  Since $R$ is connective and $\pi_0R$ is a $\bZ/p$-module we see that \[\pi_0 R\cong H_0(R;\bZ)\cong H_0(R;\bZ/p)\cong H_0(\Omega^2 S^3;\bZ/p)=\bZ/p.\] It follows that $1\in H_0(R;\bZ/p)$ must be the target of a Bockstein operation and similarly for $1\in \mathcal{A}_*$. Since $H_*f$ is a map of unital algebras commuting with the Bockstein operations we see that $H_*f$ hits the Bockstein class generating $\mathcal{A}_1$. It now follows that $f$ is a weak equivalence.

\end{proof}

This leads to the following nilpotence result, whose assumptions lie roughly in
between the nilpotence theorem and \Cref{thm:may}. It is a
generalization of Nishida's argument \cite{Nis73} on the nilpotence of order $p$ elements in the stable stems.

\begin{prop}\label{prop:simple_nilpotence}
	Suppose $R$ is an $E_2$-ring, $p$ is a prime, and $x \in \pi_*R$ is {\em simple} $p$-torsion and has nilpotent image  under the Hurewicz homomorphism $\pi_* R\rightarrow H_*(R; \mathbb{Z}/p)$. Then $x$ is nilpotent.
\end{prop}

\begin{proof}
	It suffices to show that the localization $R[x^{-1}]$ is weakly contractible. Now  $R[x^{-1}]$ is an $E_2$-ring by \Cref{thm:enloc} and $p=0\in \pi_0R[x^{-1}]$. By  \Cref{thm:hopkinsmahowald}, $R[x^{-1}]$ is an $H\bZ/p$-algebra, and in particular a generalized Eilenberg--MacLane spectrum. Since $H_*(R[x^{-1}]; \mathbb{Z}/p) = 0$, $R[x^{-1}]$ must be weakly contractible.
\end{proof}

As mentioned earlier, it is also known that any \emph{homotopy commutative} ring spectrum with $2 =
0$ is a generalized Eilenberg--MacLane spectrum; this is the main result of
\cite{W86}. We do not know if analogs of these results hold with respect to higher order
torsion. For instance, we do not know if the free $E_2$-ring with $4 = 0$ is
$K(n)$-acyclic for $0 < n < \infty$. Such a claim would strengthen our main result.
We do note that the free $E_n$-ring with $p^k=0$ is \emph{not} a generalized Eilenberg--MacLane spectrum for $n,k\geq 2$ (unpublished).


\appendix

\section{Localizations of $E_n$-ring spectra}\label{sec:localization}

In this appendix, we give a proof of the following result.
\begin{theorem} \label{thm:enloc}
Let $n \geq 2$.
Let $R$ be an $E_n$-ring and $S \subseteq \pi_* R$ be a multiplicative
subset of homogeneous elements. Then there exists a unique $E_n$-ring $S^{-1} R$ under $R$ such the map 
$R \to S^{-1}R$ induces on 
$\pi_*$ the localization map $\pi_*(R) \to S^{-1} \pi_*(R)$.
\end{theorem} 

The analog of this result for $E_1$-rings is \cite[\S 8.2.4]{Lur12}. 
Specifically, in there, it is shown:

\begin{theorem} 
\label{thm:locE1}
Let $R$ be an $E_1$-ring. Let $S \subseteq \pi_* R$ be a multiplicative subset of 
homogeneous elements that satisfies the left Ore condition \cite[Defn.~8.2.4.1]{Lur12}. 
Consider the subcategory $\mathcal{C}_S \subseteq \mathrm{Mod}(R)$ consisting of those left
$R$-modules $M$ such that for $s \in S$, multiplication by $s$ induces an
isomorphism on $\pi_*(M)$. Then:
\begin{enumerate}
\item $\mathcal{C} _S$ is closed under  arbitrary limits and colimits. 
\item The inclusion $\mathcal{C}_S \subseteq \mathrm{Mod}(R)$ admits a left
adjoint denoted $M \mapsto S^{-1 } M$.
\item 
$\mathcal{C}_S$ is generated as a localizing subcategory by $S^{-1} R$, which
is a compact object of $\mathcal{C}_S$.
\item For any $M$, the adjunction map $M \to S^{-1} M$ induces the map
$\pi_*(M) \to S^{-1}\pi_*(M)$ on homotopy groups. 
\item Let $\mathrm{Nil}_S$ be the collection of $R$-modules $M$ such that
$S^{-1}M$ is contractible. 
For each $s \in S$, let $R/s\in\mathrm{Mod}(R)$ denote the cofiber of the map $\Sigma^{|s|} R \to R$ given by right
multiplication by $s$.
Then $\mathrm{Nil}_S$ is the stable
subcategory generated under colimits by the $R/s$ for $s \in S$. \end{enumerate}
\end{theorem} 

\Cref{thm:locE1} is proved in {\em loc.~cit.}~by constructing 
the $\infty$-category of $S^{-1}R$-modules, and appealing to the
analog of the Schwede-Shipley theorem \cite{ScS03a} for compactly generated, presentable stable
$\infty$-categories. We will explain how this proof can be modified to
prove \Cref{thm:enloc}
using the $E_n$-version of the Schwede-Shipley theorem \cite[Prop.~
8.1.2.6]{Lur12}, which states that a presentable, stable
$E_{n-1}$-monoidal $\infty$-category where the tensor structure preserves
colimits in each variable is equivalent to $\mathrm{Mod}(R)$  for an $E_n$-ring $R$ if and
only if the unit $\mathbf{1}$ is a compact generator.
In addition, given $E_n$-rings $R$, and $R'$, to give a morphism of $E_n$-rings $R
\to R'$ is equivalent to giving an $E_{n-1}$-monoidal functor $\mathrm{Mod}(R)
\to \mathrm{Mod}(R')$. Here an $E_{n-1}$-monoidal
functor $\mathrm{Mod}(R) \to \mathrm{Mod}(R')$ induces a map between the endomorphism algebras of the unit objects in these two categories and this is the desired $E_n$-ring map.

\begin{proof}[Proof of \Cref{thm:enloc}] 
Since $R$ is an $E_n$-ring, the $\infty$-category $\mathrm{Mod}(R)$ is
naturally $E_{n-1}$-monoidal \cite[Prop.~6.3.5.17]{Lur12}. Given $X, Y \in \mathrm{Mod}(R)$, we let
$X \otimes_R Y $ denote the ordered $E_{n-1}$-monoidal product. It is
well-defined up to a \emph{connected} space of choices since the connected
components of the spaces in the $E_{n-1}$-operad are determined by linear
orderings for $n = 2$ (while there is only one component for $n \geq 3$). 
Moreover, since $n \geq 2$, $\pi_*(R)$ is graded-commutative and the left Ore
condition on $S $ is automatically satisfied \cite[Rem.~8.2.4.2]{Lur12}.
Therefore, it is possible to construct  a theory of $S$-localization: that is,
one can construct subcategories $\mathcal{C}_S, \mathrm{Nil}_S \subseteq	
\mathrm{Mod}(R)$ as in \Cref{thm:locE1}.

We claim that the $E_{n-1}$-monoidal structure on $\mathrm{Mod}(R)$ is compatible with
$S$-localization. In other words, there exists an $E_{n-1}$-monoidal structure
on the subcategory $\mathcal{C}_S \subseteq \mathrm{Mod}(R)$,
such that the $S$-localization $M \mapsto S^{-1}M$ is an $E_{n-1}$-monoidal
functor. 

This follows from \cite[Prop.~2.2.1.9]{Lur12} if we can prove
that, whenever we have maps of $R$-modules $M_1 \to M_2$ and $N_1 \to N_2$
that induce equivalences upon $S$-localizations, then $M_1 \otimes_R N_1 \to
M_2 \otimes_R N_2$ induces an equivalence on $S$-localizations, because then,
it will follow inductively that the operation of $S$-localization respects
arbitrary $k$-fold operations from the operad $E_{n-1}$.

Now, to say that $M_1 \to M_2$ (resp. $N_1 \to N_2$) induces an equivalence on
$S$-localizations 
is to say that the cofibers belong to $\mathrm{Nil}_S$.

It thus suffices to show that if $M, N \in \mathrm{Mod}(R)$ and \emph{one} of $M, N$ belongs
to $\mathrm{Nil}_S$, then $M \otimes_R N$ does. Suppose for definiteness that
$M$ belongs to $\mathrm{Nil}_S$. To show that $M \otimes_R N \in
\mathrm{Nil}_S$, consider the collection of all $N \in \mathrm{Mod}(R)$ such
that $M \otimes_R N \in \mathrm{Nil}_S$. This collection contains $R$, as the
unit, and it is a localizing subcategory. Therefore, it is all
of $\mathrm{Mod}(R)$ and we have proved  the existence of an
$E_{n-1}$-monoidal structure on $\mathcal{C}_S$, with $S^{-1}R$ as the unit;
as this is a compact generator, we have $\mathcal{C}_S \simeq
\mathrm{Mod}(S^{-1}R)$ as $E_{n-1}$-monoidal $\infty$-categories.
It follows from \cite[Prop.~8.1.2.6]{Lur12}
that we acquire a natural $E_{n}$-ring structure  on $S^{-1}R$.
Moreover, we obtain a natural map $R \to S^{-1} R$ of $E_n$-rings from the $E_{n-1}$-monoidal
functor $\mathrm{Mod}(R) \to \mathcal{C}_S$, by looking at endomorphisms of
the unit. 

We now prove uniqueness. From our construction of $\mathrm{Mod}(S^{-1} R)$ as a
localization of $\mathrm{Mod}(R)$, it follows that, for any $E_{n}$-ring $R'$, to give an $E_{n-1}$-monoidal
functor $\mathrm{Mod}(S^{-1}R) \to \mathrm{Mod}(R')$ is equivalent to giving an
$E_{n-1}$-monoidal functor $\mathrm{Mod}(R) \to \mathrm{Mod}(R')$ that takes $R/s$ to $0$ for
each $s \in S$. In particular, it follows that if $R'$ is an $E_{n}$-ring,
then the mapping space $\mathrm{Hom}_{E_n}(S^{-1}R, R')$ of $E_n$-ring maps
can be identified with  the union of components of 
$\mathrm{Hom}_{E_n}(R, R')$ consisting of those $E_n$-ring maps $\phi: R
\to R'$ that take each $s \in S$ to an invertible element of $\pi_*(R')$. 
Therefore, if $R'$ is an $E_{n}$-ring under $R$ such that the map $\pi_*(R)
\to \pi_*(R')$ exhibits $\pi_*(R')$ as $S^{-1}\pi_*(R)$, it follows that we
obtain a map $S^{-1}R \to R'$ of $E_{n}$-rings that is necessarily an
isomorphism.
\end{proof} 

	\bibliographystyle{amsalpha}

	\bibliography{bib/biblio}
\end{document}